\documentclass{article}
\usepackage{amsmath,amssymb}

\newtheorem{theorem}{Theorem}

\newtheorem{corollary}[theorem]{Corollary}

\newenvironment{proof}{\par\noindent{\bf Proof.}}{$\qed$\par\bigskip}
\newcommand{\qed}{\enspace\vrule  height6pt  width4pt  depth2pt}
\begin{document}

\title{Gr\"{o}bner-Shirshov bases for plactic algebras}

\author{{\L}ukasz Kubat and Jan Okni\'nski\thanks{Work supported in part by
MNiSW research grant N201 420539 (Poland).}}
\date{}
\maketitle

\begin{abstract}
A finite Gr\"obner-Shirshov basis is constructed for the plactic
algebra of rank $3$ over a field $K$. It is also shown that
plactic algebras of rank exceeding $3$ do not have finite
Gr\"obner-Shirshov bases associated to the natural
degree-lexicographic ordering on the corresponding free algebra.
The latter is in contrast with the case of a strongly related
class of algebras, called Chinese algebras, considered in
\cite{chen}.
\end{abstract}

Let $P_{n}$ denote the plactic algebra of rank $n\geq 3$ over a
filed $K$. So $P_{n}=K[M_{n}]$ is the monoid algebra over $K$ of
the plactic monoid $M_{n}$ of rank $n$, which is defined by the
following finite presentation:
$$M_{n}=\langle x_{1},\ldots ,x_{n}\mid R\rangle ,$$
where $R$ is the set of all relations (called Knuth relations) of
the form
$$x_{j}x_{i}x_{k}= x_{j}x_{k}x_{i} , \,\,\, x_{i}x_{k}x_{j}= x_{k}x_{i}x_{j}
\mbox{ for } i<j<k$$ and $$x_{j}x_{i}x_{i}= x_{i}x_{j}x_{i},
\,\,\,  x_{j}x_{i}x_{j}= x_{j}x_{j}x_{i}  \mbox{ for } i<j.$$ The
origin of the plactic monoid stems from Schensted's algorithm that
was developed in order to determine the maximal length of a
non-increasing subsequence and the maximal length of a decreasing
subsequence in any finite sequence with elements in the set $\{
1,2,\ldots ,n\}$, see \cite{las-schut}.  Combinatorics of
$M=M_{n}$ was thoroughly studied in \cite{las-schut}, see also
\cite{las-lec}. In particular, it is known that elements of $M$
admit a canonical normal form, expressed in terms of the
associated Young tableaux. Later, deep applications of the plactic
monoid to problems in representation theory, algebraic
combinatorics, theory of quantum groups and relations to some
other important areas of mathematics were discovered. We refer to
\cite{fulton},\cite{las-lec} and \cite{leclerc} for a presentation
of these topics. Some ring theoretical properties of the plactic
algebra were described in \cite{cedo-okn}. The purpose of this
note is to study Gr\"obner-Shirshov bases of $K[M]$ related to the
natural order on the associated free algebra.

Let $X$ be the free monoid of rank $n$, with  free generators also
denoted by $x_{1},\ldots ,x_{n}$. Then $X$ is equipped with the
degree-lexicographic order extending the following order on the
generating set of $X$: $x_{1}<x_{2}<\cdots <x_{n}$. For
simplicity, the generators $x_{i}$ will be also denoted by $i$, if
unambiguous. This natural order is inherent in the nature of the
plactic monoid, originally developed for the combinatorial problem
mentioned above.

Our first result reads as follows.

\begin{theorem} \label{n3}
If $n=3$ then $K[M]$ has a finite Gr\"{o}bner-Shirshov basis. Namely,
the following elements, viewed as elements of the free algebra
$K[X]$, form such a basis:
\begin{enumerate}
\item[i)] 332 - 323
\item[ii)] 322 - 232
\item[iii)] 331 - 313
\item[iv)]  311 - 131
\item[v)] 221 - 212
\item[vi)]  211 - 121
\item[vii)] 231 - 213
\item[viii)]  312 - 132
\item[ix)]  3212 - 2321
\item[x)] 32131 - 31321
\item[xi)] 32321 - 32132.
\end{enumerate}
\end{theorem}
\begin{proof}
Recall that by a reduction we mean a replacement of a subword of a
given word $w$ of the free monoid $X$ that is the leading term of
an element $f\in K[X]$ listed in i)-xi) by a subword that is equal
to the remaining monomial $w'$ of $f$. For example,
$1321223\rightarrow 1232123$ is a reduction using the polynomial
$f = 3212 - 2321$ listed in ix) above. First, we list all
ambiguities between two types of reductions that can occur in the
process of bringing $w$ to a form that cannot reduced anymore
(notice that such a form exists because the defining relations of
$M$ are homogeneous):
\begin{enumerate}
\item $(332)2 = 3(322)$

\item $(332)21 = 33(221)$

\item $(332)11=33(211)$

\item $(332)31 = 33(231)$

\item $(332)12 = 3(3212)$

\item $(332)131 = 3(32131)$

\item $(332)321 = 3(32321)$

\item $(322)1 = 3(221)$

\item $(322)21 = 32(221)$

\item $(322)11 = 32(211)$

\item $(322)31 = 32(231)$

\item $(331)1=3(311)$

\item $(331)2 = 3(312)$

\item $2(311)=(231)1 $

\item $321(311)=(32131)1$

\item $(221)1 = 2(211)$

\item $31(221) = (312)21$

\item $321(221) = (3212)21$

\item $31(211) = (312)11$

\item $321(211) = (3212)11$

\item $323(211)=(32321)1 $

\item $(231)2 = 2(312)$

\item $31(231)=(312)31$

\item $321(231) = (3212)31$

\item $321(312) = (32131)2$

\item $32(3212) = (32321)2$

\item $32(32131) = (32321)31$
\end{enumerate}
Next, we show that all ambiguities can be resolved, using the
reductions provided by the elements i)-xi) listed above (see the
diamond lemma in \cite{bergman}, see also \cite{cohn}).
\begin{enumerate}
\item
\noindent $(332)2 \rightarrow 3232$ (using i))

\noindent $3(322)\rightarrow 3232$ (using ii))
\item
\noindent $(332)21 \rightarrow 32321$ (using i))

\noindent $33(221)\rightarrow 33213 \rightarrow 32313 $ (using v),
i))

\item
\noindent $(332)11\rightarrow  32311 \rightarrow  32131
\rightarrow  31321 $ (using i), vii), x))

\noindent $33(211)\rightarrow  33121 \rightarrow  31321$   (using
vi), iii))

\item
\noindent $(332)31 \rightarrow 32331\rightarrow 32313 $ (using i),
iii))

\noindent $ 33(231)\rightarrow 33213 \rightarrow 32313 $ (using
vii), i))
\item
\noindent $(332)12 \rightarrow 32312 \rightarrow 32132 $ (using
i), vii))

\noindent $3(3212)\rightarrow 32321 \rightarrow 32132$ (using ix),
xi))
\item
\noindent $(332)131 \rightarrow 323131 \rightarrow 321331
\rightarrow 321313 \rightarrow 313213 $ (using i), vii), iv), x))

\noindent $ 3(32131)\rightarrow 331321 \rightarrow 313321
\rightarrow 313231 \rightarrow 313213$ (using xi), iii), i), vii))
\item
\noindent $(332)321 \rightarrow 323321 \rightarrow 323231
\rightarrow 323213 \rightarrow 321323 $ (using i), i), vii), xi))

\noindent $ 3(32321)\rightarrow332132 \rightarrow 323132
\rightarrow 321332 \rightarrow 321323 $ (using xi), i), vii), i))
\item
\noindent $(322)1 \rightarrow 2321$ (using ii))

\noindent $3(221)\rightarrow 3212 \rightarrow 2321$ (using v),
ix))

\item
\noindent $(322)21 \rightarrow 23221 \rightarrow 22321 $ (using
ii), ii))

\noindent $ 32(221) \rightarrow 32212 \rightarrow 23212
\rightarrow 22321$ (using v), ii), ix))

\item
\noindent $(322)11\rightarrow 23211$ (using i))

\noindent $ 32(211)\rightarrow 32121 \rightarrow 23211$ (using
vi), ix))
\item
\noindent $(322)31\rightarrow 23231 \rightarrow 23213$ (using ii),
vii))

\noindent $ 32(231)\rightarrow 32213 \rightarrow 23213$ (using
vii), ii))
\item
\noindent $(331)1\rightarrow 3131$ (using iii))

\noindent $3(311)\rightarrow 3131$ (using iv))
\item
\noindent $(331)2 \rightarrow 3132$ (using iii))

\noindent $3(312)\rightarrow 3132$ (using viii))
\item
\noindent $2(311)\rightarrow 2131$ (using iv))

\noindent $(231)1\rightarrow 2131$ (using vii))
\item
\noindent $321(311) \rightarrow 321131 \rightarrow
312131\rightarrow 132131 \rightarrow 131321$ (using iv), vi),
viii), x))

\noindent $(32131)1 \rightarrow 313211 \rightarrow
313121\rightarrow 311321 \rightarrow 131321$ (using x), vi),
viii), iv))
\item
\noindent $(221)1 \rightarrow 2121$ (using v))

\noindent $2(211)\rightarrow 2121$ (using vi))
\item
\noindent $31(221) \rightarrow 31212 \rightarrow 13212$ (using v),
viii))

\noindent $ (312)21\rightarrow 13221 \rightarrow 13212 $ (using
viii), v))
\item
\noindent $321(221) \rightarrow 321212 \rightarrow 232112
\rightarrow 231212 \rightarrow 213212 \rightarrow 212321$ (using
v), ix), vi), vii), ix))

\noindent $(3212)21\rightarrow 232121 \rightarrow 223211
\rightarrow 223121 \rightarrow 221321 \rightarrow 212321 $ (using
ix),ix), vi), vii), v))
\item
\noindent $31(211) \rightarrow 31121 \rightarrow 13121$ (using
vi), iv))

\noindent $(312)11\rightarrow 13211 \rightarrow 13121 $ (using
viii), vi))
\item
\noindent $321(211)\rightarrow 321121 \rightarrow 312121
\rightarrow 132121 \rightarrow 123211 \rightarrow 123121
\rightarrow 121321 $ (using vi), vi), viii), ix), vi), viii))

\noindent $ (3212)11\rightarrow 232111 \rightarrow 231211
\rightarrow 213211 \rightarrow 213121 \rightarrow 211321
\rightarrow 121321$ (using ix), vi), vii), vi), viii), vi))
\item
\noindent $323(211)\rightarrow 323121 \rightarrow 321321$ (using
vi), vii))

\noindent $(32321)1 \rightarrow 321321$ (using xi))
\item
\noindent $(231)2 \rightarrow 2132$ (using vii))

\noindent $2(312)\rightarrow 2132$ (using viii))
\item
\noindent $31(231)\rightarrow 31213 \rightarrow 13213 $ (using
vii), viii))

\noindent $(312)31\rightarrow 13231 \rightarrow 13213$ (using
viii), vii))
\item
\noindent $321(231) \rightarrow 321213 \rightarrow 232113
\rightarrow 231213 \rightarrow 213213$ (using vii), ix), vi),
vii))

\noindent $(3212)31\rightarrow 232131 \rightarrow 231321
\rightarrow 213321 \rightarrow 213231 \rightarrow 213213$ (using
ix), x), vii), i), vii))
\item
\noindent $321(312) \rightarrow 321132 \rightarrow 312132
\rightarrow 132132$ (using viii), vi), viii))

\noindent $(32131)2\rightarrow 313212 \rightarrow 312321
\rightarrow 132321 \rightarrow 132132$ (using x), ix), viii), xi))

\item
\noindent $32(3212)\rightarrow 322321 \rightarrow 232321
\rightarrow 232132 $ (using ix), ii), xi))

\noindent $ (32321)2\rightarrow 321322 \rightarrow 321232
\rightarrow 232132$ (using xi), ii), ix))

\item
\noindent
 $(32321)31  \rightarrow32132 31\rightarrow 321 321 3$
(using x), vii))

\noindent $32(32131)\rightarrow 32 31321\rightarrow 321 3 321
\rightarrow 321 3 231 \rightarrow 3213213$ (using xi), vii), i),
v))
\end{enumerate}
The result follows.
\end{proof}

Notice that the above result may be reformulated to say that $M$
admits a complete rewriting system, see \cite{otto}.

Let $v,u$ be elements of the free monoid $X$ with free generators
$\{1,\ldots ,n\}$. If $v=v_{m}\cdots v_{1},u=u_{q}\cdots u_{1}\in
X$, $v_{j},u_{i}\in \{1,\ldots ,n\}$, then we write $v\prec u$ if
$m\geq q$ and $v_{i} \leq u_{i}$ for $i=1,\ldots , q$. Recall from
\cite{las-lec} that every $w\in M$ can be uniquely written in the
form $w=w_{1}w_{2}\cdots w_{k}$ for some decreasing words
$w_{1},\ldots ,w_{k}$ such that $w_{i}\prec w_{i+1}$ for
$i=1,\ldots ,k-1$. This is called the standard tableaux form of
$w$. For example, $(421)^{2}(31)(32)2^{3}4$ is an element of
$M=M_{4}$ written in this form.

The following is an easy consequence of Theorem~\ref{n3}.

\begin{corollary}
The Gr\"obner-Shirshov basis found above leads to the following
normal forms of elements of the plactic monoid of rank $3$
$$(1)^{i} (21)^{j} (2)^{k} (321)^{l} (32)^{m} (3)^{q}$$
or
$$(1)^{i} (21)^{j} (31)^{k} (321)^{l} (32)^{m} (3)^{q}$$
for non-negative integers $i,j,k,l,m,q$.
\end{corollary}
\begin{proof}
It is clear that words of the above two types cannot be reduced,
using the reductions described in Theorem~\ref{n3}. Hence, it
remains to show that every word can be reduced to one of the above
forms; or, that every minimal word $w\in X$ (a word that cannot be
reduced) is of one of the above forms. This can be easily seen
because if $w\in \langle 1,2\rangle $ then $w$ can be reduced to a
word of the form $1^{i} (21)^{j}2^{k}$. Otherwise, write $w=u3v$,
where $u,v\in X$. If $v\in 1M$ then $u\not \in M2$, so $u$ can be
reduced to a word of the form $1^{i} (21)^{j}$. Otherwise, $u$ can
be reduced to $1^{i} (21)^{j}2^{k}$. Next, it is easy to see that
$3v$ is of the form $(31)^{p}(321)^{l} (32)^{m} (3)^{q}$. Since
$231$ cannot be a subword of $w$, the assertion follows.

Another way of proving the claim is to show that the standard
tableaux forms of the elements listed in the statement (they all
must be different) exhaust all tableaux forms of elements of $M$.
This is an easy consequence of the algorithm that allows to bring
any word to a word in the tableau form, see \cite{las-lec}
\end{proof}

We conclude with the following surprising observation.

\begin{theorem}
If $n>3$ then every Gr\"{o}bner-Shirshov basis of $K[M]$ (associated
to the degree-lexicographic ordering of $M$) is infinite.
\end{theorem}
\begin{proof}
Consider the words $w_{i}=32 3^{i} 431 \in X$, for $i=1,2,\ldots$.
Since we have the following equalities in $M$:
$$323=332, \, 32431=34231=34213=32413=32143, \, 3213=3321$$
it follows that $32 3^{i} 431=w_{i}=3213^{i}43,$ holds also in
$M$.

Let $a=32 3^{i} 43, b=2 3^{i} 431\in X$. We claim that $a$ is the
minimal word in $X$ among all words that represent $32 3^{i} 43$
as an element of $M$. We also claim that  $b$ is the minimal word
in $X$ among all words that represent $2 3^{i} 431$ as an element
of $M$. Then it is clear that, in order to reduce the word $32
3^{i} 431$, we have to include the reduction $32 3^{i} 431
\rightarrow 3 213^{i}43$ to the set of allowed reductions. Since
$i\geq 1$ is arbitrary, the set of reductions must be infinite,
and the assertion follows.

In order to verify the claims, notice first that the defining
relations do not allow to rewrite $32 3^{i} 43$ in $M$ in the form
$2v$ form some $v\in X$, or in the form $u4$ for some $u\in X$. It
follows easily that $32 3^{i} 43$ is a minimal word in its class
in $M$.

Next, consider the word $b=2 3^{i} 431$. Suppose $b\in 1M$. The
only defining relations that can be used to bring $1$ to the first
position in a presentation of $b$ are $312=132, 412=142$ or
$413=143$. In the corresponding cases we would get
$b=1323^{j}43^{i-j}, b= 142 3^{i+1}, b=143 3^{j}23^{i-j}$, for
some $j\geq 0$, respectively. Since the maximal length of a
decreasing subsequence of the given word is an invariant of the
plactic class of this word, see \cite{las-lec}, this implies that
$b=143 3^{j}23^{i-j}$. However, the standard tableaux form of the
latter element is easily seen to be $(431)2 3^{i}$, while the
standard tableaux form of $2 3^{i} 431$ is $(421)3^{i+1}$,
\cite{las-lec}. This contradiction shows that $b\notin 1M$.
Suppose $b \in 21M$. Then $b=213^{j}43^{i-j+1}$ in $M$, for soome
$j\geq 0$, again a contradiction because the latter has no
decreasing subsequence of length three. So a minimal word that
represents $b$ in $M$ should start with $23$ and hence it must be
of the form $23^{i}431$ (because it should have a decreasing
subsequence of length three). The result follows.
\end{proof}

The above result is in contrast with the corresponding result for
the strongly related class of monoids, also defined by homogeneous
monoid presentations and with all defining relations of degree $3$
- the so called Chinese monoids. The latter class was introduced
and its combinatorial properties were studied in \cite{chinese}.
It was shown in \cite{chen} that the Gr\"{o}bner-Shirshov basis of the
Chinese algebra of any rank $n\geq 1$ is finite. Notice that the
Chinese algebra of rank $n$ has the same growth function (see
\cite{krause}) as the plactic algebra of rank $n$ and its elements
also admit a canonical form expressed in terms of certain
tableaux, \cite{duchamp}. Moreover, if $n<3$, then the two
algebras coincide.

\vspace{30pt}
 \noindent \begin{tabular}{llllllll}
 {\L}ukasz Kubat && Jan Okni\'nski\\
 Institute of Mathematics && Institute of Mathematics \\
 Polish Academy of Sciences && Warsaw University \\
 \'Sniadeckich 8 &&  Banacha 2\\
 00-956 Warsaw, Poland && 02-097 Warsaw, Poland \\
  kubat.lukasz@gmail.com && okninski@mimuw.edu.pl
\end{tabular}

\end{document}